\documentclass{amsart}

\theoremstyle{definition}

\usepackage{amsmath}
\usepackage{amssymb}
\usepackage{amsthm}

\DeclareMathOperator{\tr}{tr}

\DeclareMathOperator{\Vol}{Vol}
\DeclareMathOperator{\dvol}{dvol}

\DeclareMathOperator{\Ric}{Ric}

\DeclareMathOperator{\Rm}{Rm}

\DeclareMathOperator{\End}{End}

\DeclareMathOperator{\Sym}{Sym}

\DeclareMathOperator{\tf}{tf}
\DeclareMathOperator{\rank}{rank}

\newcommand{\oJ}{\overline{J}}

\newcommand{\oP}{\overline{P}}

\newcommand{\oX}{\overline{X}}

\newcommand{\og}{\overline{g}}

\newcommand{\odelta}{\overline{\delta}}

\newcommand{\oDelta}{\overline{\Delta}}

\newcommand{\onabla}{\overline{\nabla}}

\newcommand{\hg}{\widehat{g}}

\newcommand{\lp}{\langle}
\newcommand{\rp}{\rangle}
\newcommand{\lv}{\lvert}
\newcommand{\rv}{\rvert}

\newcommand{\mB}{\mathcal{B}}
\newcommand{\mC}{\mathcal{C}}

\newcommand{\mE}{\mathcal{E}}
\newcommand{\mF}{\mathcal{F}}

\newcommand{\mL}{\mathcal{L}}

\newcommand{\mQ}{\mathcal{Q}}

\newcommand{\mS}{\mathcal{S}}

\newcommand{\kC}{\mathfrak{C}}

\newcommand{\bN}{\mathbb{N}}

\newcommand{\bR}{\mathbb{R}}

\newcommand{\suchthat}{\mathrel{}\middle|\mathrel{}}

\def\sideremark#1{\ifvmode\leavevmode\fi\vadjust{\vbox to0pt{\vss
 \hbox to 0pt{\hskip\hsize\hskip1em
 \vbox{\hsize3cm\tiny\raggedright\pretolerance10000
 \noindent #1\hfill}\hss}\vbox to8pt{\vfil}\vss}}}

\newcommand{\comment}[1]{}

\newtheorem{thm}{Theorem}[section]
\newtheorem{prop}[thm]{Proposition}
\newtheorem{lem}[thm]{Lemma}
\newtheorem{cor}[thm]{Corollary}

\theoremstyle{definition}

\theoremstyle{remark}
\newtheorem{remark}[thm]{Remark}

\numberwithin{equation}{section}

\begin{document}

\title{Boundary operators associated to the $\sigma_k$-curvature}
\author{Jeffrey S. Case}
\address{109 McAllister Building \\ Penn State University \\ University Park, PA 16802}
\email{jscase@psu.edu}
\author{Yi Wang}
\thanks{YW was partially supported by NSF Grant No.\ DMS-1612015}
\address{3400 N. Charles St. \\ 216 Krieger Hall \\ Mathematics Department \\ Baltimore, MD 21218}
\email{ywang@math.jhu.edu}
\begin{abstract}
 We study conformal deformation problems on manifolds with boundary which include prescribing $\sigma_k\equiv0$ in the interior.  In particular, we prove a Dirichlet principle when the induced metric on the boundary is fixed and an Obata-type theorem on the upper hemisphere.  We introduce some conformally covariant multilinear operators as a key technical tool.
\end{abstract}
\maketitle

\section{Introduction}
\label{sec:intro}

Let $(X^{n+1},g_0)$ be a compact Riemannian manifold.  The Ricci decomposition
\[ \Rm = W + P\wedge g_0 \]
of the Riemann curvature tensor $\Rm$ into the conformally invariant Weyl curvature $W$ and the Kulkarni--Nomizu product of the Schouten tensor
\[ P = \frac{1}{n-1}\left(\Ric - \frac{R}{2n}g_0\right) \]
and the metric $g$ implies that the behavior of the full Riemann curvature tensor under conformal deformation is completely controlled by the Schouten tensor.  For this reason, Viaclovsky initiated~\cite{Viaclovsky2000} the study of the conformal properties of the \emph{$\sigma_k$-curvatures} $\sigma_k$; i.e.\ the $k$-th elementary symmetric functions of the Schouten tensor.  Note that $\sigma_1=R/2n$ is proportional to the scalar curvature.  Crucially, the $\sigma_k$-curvatures are variational if and only if $k\leq 2$ or $g$ is locally conformally flat~\cite{BransonGover2008,Viaclovsky2000}.  In particular, if $k\leq2$ or $g$ is locally conformally flat, then the total $\sigma_k$-curvature functional, $\mF(g):=\int_X \sigma_kg\,\dvol_g$, is such that
\begin{equation}
 \label{eqn:closed_variational}
 \left.\frac{d}{dt}\right|_{t=0}\mF_k\left(e^{2t\Upsilon}g\right) = (n+1-2k)\int_X \Upsilon\sigma_k^g\,\dvol_g
\end{equation}
for all metrics $g$ in the conformal class $[g_0]$ of $g_0$ and all $\Upsilon\in C_0^\infty(X)$.  Equation~\eqref{eqn:closed_variational} realizes $\sigma_k$ as the conformal gradient of $\mF_k$ when $n+1\not=2k$; Brendle and Viaclovsky found~\cite{BrendleViaclovsky2004} a different functional with conformal gradient $\sigma_k$ when $n+1=2k$.

When $X^{n+1}$ has nonempty boundary $M^n$, S.\ Chen introduced~\cite{Chen2009s} the $H_k$-curvatures as a family of invariants on $M$ which are polynomial in the restriction $P\rv_{TM}$ of the Schouten tensor to $M$ and the second fundamental form $A$ of $M$.  For example, $H_1=\frac{1}{n}\tr A$ is the mean curvature; see Section~\ref{sec:bg} for the general formula for $H_k$.  A key property of the $H_k$-curvatures is that they enable the study of conformal deformations of the $\sigma_k$-curvature on manifolds with boundary by variational methods: If $k\leq 2$ or $g_0$ is locally conformally flat, then the functional
\[ \mS_k(g) := \int_X \sigma_k^g\,\dvol_g + \oint_M H_k^g\,\dvol_{i^\ast g}, \]
where $i\colon M\to X$ is the inclusion mapping, satisfies
\begin{equation}
 \label{eqn:general_variation}
 \left.\frac{d}{dt}\right|_{t=0}\mS_k\left(e^{2t\Upsilon}g\right) = (n+1-2k)\left[ \int_X \Upsilon\sigma_k^g\,\dvol_g + \oint_M \Upsilon H_k^g\,\dvol_{i^\ast g} \right]
\end{equation}
for all $g\in[g_0]$ and all compactly-supported $\Upsilon\in C^\infty(X)$.  Equation~\eqref{eqn:general_variation} realizes $(\sigma_k,H_k)$ as the conformal gradient of $\mS_k$ when $n+1\not=2k$; we provide in Proposition~\ref{prop:boundary_functional} below a different functional with conformal gradient $(\sigma_k,H_k)$ when $n+1=2k$.

The results of this article provide existence and uniqueness results for certain problems involving conformal deformations of a compact Riemannian manifold $(X^{n+1},g_0)$ with nonempty boundary to a Riemannian manifold for which $\sigma_k^g$ vanishes identically.  Our results use variational methods, and hence we always assume that $k\leq2$ or $g_0$ is locally conformally flat.  In order to use elliptic methods, we always restrict our attention to the $C^{1,1}$-closures of the cones
\[ \Gamma_k^+ = \left\{ g\in[g_0] \suchthat \sigma_j^g>0 \text{ for all $1\leq j\leq k$} \right\} . \]
A key point is that $\sigma_k^g$ is elliptic (resp.\ degenerate elliptic) in the cone $\Gamma_k^+$ (resp.\ $\overline{\Gamma_k^+}$).

Our first main result is the following Dirichlet principle for our deformation problem:

\begin{thm}
 \label{thm:intro/inequality}
 Let $k\in\bN$ and let $(X^{n+1},g_0)$ be a compact Riemannian manifold with nonempty boundary $M=\partial X$ such that $n+1\not=2k$ and $g_0\in\Gamma_k^+$.  If $k\geq3$, assume additionally that $g_0$ is locally conformally flat.  Set
 \begin{equation}
  \label{eqn:intro/dirichlet}
  \mC_k = \left\{ u\in C^\infty(X) \suchthat g_u:=e^{2u}g_0\in\Gamma_k^+, u\rv_M\equiv0 \right\}
 \end{equation}
 and let $\overline{\mC_k}$ be the $C^{1,1}$-closure of $\mC_k$.  Then there is a unique $u_0\in\overline{\mC_k}$ such that $\sigma_k^{g_{u_0}}\equiv0$.  Moreover,
 \begin{equation}
  \label{eqn:intro/inequality}
  \mS_k(g_u) \geq \mS_k(g_{u_0})
 \end{equation}
 for all $u\in\overline{\mC_k}$ with equality if and only if $u=u_0$.
\end{thm}

Our proof of Theorem~\ref{thm:intro/inequality} proceeds by finding conformally covariant $(2k-1)$-linear operators $L_k\colon\left(C^\infty(X)\right)^{2k-1}\to C^\infty(X)$ and $B_k\colon\left(C^\infty(X)\right)^{2k-1}\to C^\infty(M)$ such that
\begin{enumerate}
 \item $L_k(1,\dotsc,1)=c_{n,k}\sigma_k$ for an explicit constant $c_{n,k}$ which depends only on $n$ and $k$ and vanishes if and only if $n+1=2k$;
 \item $B_k(1,\dotsc,1)=c_{n,k}H_k$ for $c_{n,k}$ the same constant as above; and
 \item the functional
 \[ \mQ_k(u_1,\dotsc,u_{2k}) := \int_X u_1\,L_k(u_2,\dotsc,u_{2k})\,\dvol_g + \oint_M u_1\,B_k(u_2,\dotsc,u_{2k})\,\dvol_{i^\ast g} \]
 is symmetric in $(u_1,\dotsc,u_{2k})\in\left(C^\infty(X)\right)^{2k}$.
\end{enumerate}
In the special case $k=1$, we recover the conformal Laplacian as $L_1$ and the conformal Robin operator as $B_1$ (cf.\ \cite{Escobar1988,Escobar1992a}).  In general, the above properties allow us to rewrite the functional $\mS$ on $[g_0]$ in terms of the ``energy functional''
\[ \mE_k(u) := \int_X u\,L_k(u,\dotsc,u)\,\dvol_g + \oint_M u\,B_k(u,\dotsc,u)\,\dvol_{i^\ast g} . \]
via the identity
\[ c_{n,k}\mS_k\left(u^{\frac{4k}{n+1-2k}}g_0\right) = \mE_k(u) . \]
In this formulation, Equation~\eqref{eqn:intro/inequality} is equivalent to the Dirichlet-type principle
\begin{equation}
 \label{eqn:intro/inequality_function}
 \mE_k(u) \geq \mE_k(u_0)
\end{equation}
for all $u\in\overline{\mC_k}$ with equality if and only if $u=u_0$.  The proof of~\eqref{eqn:intro/inequality_function}, and hence~\eqref{eqn:intro/inequality}, uses the symmetry of $\mQ_k$, and is completely analogous to the corresponding Dirichlet principle for the $k$-Hessian equation~\cite{CaseWang2016}.

Theorem~\ref{thm:intro/inequality} and its reformulation as the functional inequality~\eqref{eqn:intro/inequality_function} suggests an interesting new construction of fully nonlinear nonlocal conformally covariant operators.  Specifically, under the hypotheses of Theorem~\ref{thm:intro/inequality}, the generalized Dirichlet-to-Neumann operator $\mB_k(f) := B_k(u_f)$ is well-defined, where $u_f$ is the unique solution of
\[ \begin{cases}
    \sigma_k^{u^{\frac{4k}{n+1-2k}}g} = 0, & \text{in $X$}, \\
    u = f, & \text{on $M$} ,
   \end{cases} \]
under suitable hypotheses on the domain.  In this form, \eqref{eqn:intro/inequality_function} becomes the Sobolev trace-type inequality
\begin{equation}
 \label{eqn:intro/sobolev_trace}
 \mE_k(u) \geq \oint_M f\,\mB_k(f)\,\dvol_{i^\ast g}
\end{equation}
for all suitable extensions of $f$.  Indeed, Guillarmou and Guillop\'e showed~\cite{GuillarmouGuillope2007} that $\mB_1$ is the fractional GJMS operator of order $1$, as defined by Graham and Zworski~\cite{GrahamZworski2003}, whence~\eqref{eqn:intro/sobolev_trace} essentially provides a norm computation for the Sobolev trace embedding $H^1(X)\subset H^{1/2}(M)$.  For this reason, we regard~\eqref{eqn:intro/sobolev_trace} as a norm computation for part of the Sobolev trace embedding $W^{1,2k}(X) \subset W^{1-1/2k,2k}(M)$.

The Sobolev embedding $W^{1-1/2k,2k}(M^n)\subset L^{\frac{2nk}{n+1-2k}}(M^n)$ and the previous paragraph together suggest that the following sharp Sobolev trace-type inequality should be true: If $n+1>2k$, then for any metric $g\in\overline{\Gamma_k^+}$ conformal to the round metric $d\theta_+^2$ on the upper hemisphere $S_+^{n+1}$, it holds that
\begin{equation}
 \label{eqn:intro/sobolev_conjecture}
 \int_{S_+^{n+1}} \sigma_k^g\,\dvol_g + \oint_{S^n} H_k^g\,\dvol_{i^\ast g} \geq C_{n,k}\left(\Vol_g(S^n)\right)^{\frac{n+1-2k}{n}}
\end{equation}
with equality if and only if $g$ is flat, where $C_{n,k}>0$ is an explicitly computable constant.  The validity of~\eqref{eqn:intro/sobolev_conjecture} is further substantiated by the sharp Sobolev-type inequalities involving the $\sigma_k$-curvature on round spheres proven by Guan and Wang~\cite{GuanWang2004}.

As a step towards proving the validity of~\eqref{eqn:intro/sobolev_conjecture}, we establish the following partial classification of the critical points of the functional
\[ g \mapsto \mS_k(g)\Vol_g(M)^{-\frac{n+1-2k}{n}} , \]
or equivalently, metrics for which $\sigma_k^g\equiv0$ and $H_k^g$ is constant.

\begin{thm}
 \label{thm:intro/obata}
 Let $\Gamma_k^+$ be the positive $k$-cone on the round $(n+1)$-dimensional hemisphere $S_+^{n+1}$.  Suppose that $g\in\overline{\Gamma_k^+}$ is such that $\sigma_k^g\equiv0$ and $H_k^g$ is a positive constant along $S^n=\partial S_+^{n+1}$.  Suppose additionally that
 \begin{equation}
  \label{eqn:intro/obata_assumption}
  \sup_{S^n} H^g \leq (k+1)\inf_{S^n} H^g .
 \end{equation}
 Then $g$ is flat and $g\rv_{S^n}$ is round.
\end{thm}

When $k=1$, the pinching assumption~\eqref{eqn:intro/obata_assumption} automatically holds since $H_1$ is the mean curvature $H$, so that Theorem~\ref{thm:intro/obata} recovers an Obata-type theorem of Escobar~\cite{Escobar1988}.  We expect that the pinching assumption~\eqref{eqn:intro/obata_assumption} is not necessary for general $k$.  Our proof of Theorem~\ref{thm:intro/obata} adapts insights from Escobar's proof as well as from Viaclovsky's proof~\cite{Viaclovsky2000} of a similar Obata-type theorem on $S^n$.

This article is organized as follows: In Section~\ref{sec:bg} we recall some basic facts about the $\sigma_k$- and $H_k$-curvatures, and provide a more refined understanding of their combined variational properties.  In Section~\ref{sec:boundary} we define the functionals $L_k$ and $B_k$, prove their key symmetry and conformal covariance properties, and establish a useful nonnegativity result for the functional $\mQ_k$.  In Section~\ref{sec:inequality} we prove Theorem~\ref{thm:intro/inequality}.  In Section~\ref{sec:obata} we prove Theorem~\ref{thm:intro/obata}.

\section{Background}
\label{sec:bg}

In this section we recall some background relating to elementary symmetric functions on matrices and on manifolds.  We are particularly interested in their convexity and variational properties.  On manifolds with boundary, one must introduce suitable invariants on the boundary to study such properties.  Such invariants, called the $H_k$-curvatures, were introduced by Chen~\cite{Chen2009s}, though their variational properties were not fully detailed.  Lemma~\ref{lem:conformal_variation} addresses this point by computing the conformal linearization of the $H_k$-curvature.  This enables us to specify a functional with conformal gradient $(\sigma_k;H_k)$ on compact Riemannian manifolds $(X,g)$ with boundary in the critical dimension $\dim X=2k$.

\subsection{Elementary symmetric functions of symmetric matrices}
\label{subsec:bg/sk}

For $k\in\bN$, the \emph{$k$-th elementary symmetric function} of a $d\times d$-symmetric matrix $B\in\Sym_d$ is
\[ \sigma_k(B) := \sum_{i_1<\dotsb<i_k} \lambda_{i_1}\dotsm\lambda_{i_k} , \]
where $\lambda_1,\dotsc,\lambda_d$ are the eigenvalues of $B$.  The $k$-th elementary symmetric function can also be computed without knowledge of the eigenvalues of $B$ via the formula
\begin{equation}
 \label{eqn:sym_sigmak}
 \sigma_k(B) = \frac{1}{k!}\delta_{i_1\dotsb i_k}^{j_1\dotsb j_k} B_{j_1}^{i_1} \dotsm B_{j_k}^{i_k} ,
\end{equation}
where $\delta_{i_1\dotsb i_k}^{j_1\dotsb j_k}$ denotes the generalized Kronecker delta,
\[ \delta_{i_1\dotsb i_k}^{j_1\dotsb j_k} = \begin{cases}
                                             1, & \text{if $(i_1,\dotsc,i_k)$ is an even permutation of $(j_1,\dotsc,j_k)$,} \\
                                             -1, & \text{if $(i_1,\dotsc,i_k)$ is an odd permutation of $(j_1,\dotsc,j_k)$,} \\
                                             0, & \text{otherwise,}
                                            \end{cases} \]
and Einstein summation convention is employed.  The \emph{$k$-th Newton tensor} of $B$ is
\begin{equation}
 \label{eqn:sym_Tk}
 T_k(B)_i^j = \frac{1}{k!}\delta_{ii_1\dotsb i_k}^{jj_1\dotsb j_k} B_{j_1}^{i_1} \dotsm B_{j_k}^{i_k} .
\end{equation}
It is clear from both~\eqref{eqn:sym_sigmak} and~\eqref{eqn:sym_Tk} that $\sigma_k(B)$ and $T_k(B)$ are homogeneous polynomials of degree $k$ in $B$.  In particular, we may polarize both $\sigma_k(B)$ and $T_k(B)$.  For our purposes, we require only the mixed symmetric functions and Newton tensors: Given nonnegative integers $k,\ell$ with $\ell\leq k$ and $d\times d$-symmetric matrices $B,C$, we define
\begin{align*}
 \sigma_{k,\ell}(B,C) & = \frac{1}{k!}\delta_{i_1\dotsb i_k}^{j_1\dotsb j_k} B_{j_1}^{i_1} \dotsm B_{j_\ell}^{i_\ell} C_{j_{\ell+1}}^{i_{\ell+1}} \dotsm C_{j_k}^{i_k} , \\
 T_{k,\ell}(B,C)_i^j & = \frac{1}{k!}\delta_{ii_1\dotsb i_k}^{jj_1\dotsb j_k} B_{j_1}^{i_1} \dotsm B_{j_\ell}^{i_\ell} C_{j_{\ell+1}}^{i_{\ell+1}} \dotsm C_{j_k}^{i_k} , \\
\end{align*}
That is, $\sigma_{k,\ell}(B,C)$ (resp.\ $T_{k,\ell}(B,C)$) is the polarization of $\sigma_k$ (resp.\ $T_k$) evaluated at $\ell$ factors of $B$ and $k-\ell$ factors of $C$.

The elementary symmetric functions and Newton tensors are particularly well-behaved within the G{\aa}rding cones.  For example, the \emph{positive $k$-cone} is
\[ \Gamma_k^+ := \left\{ B \in \Sym_n \suchthat \sigma_1(B),\dotsc,\sigma_k(B)>0 \right\} . \]
There are a number of important facts about elements of $\Gamma_k^+$, among them the following (see~\cite{CaffarelliNirenbergSpruck1985} for proofs):
\begin{enumerate}
 \item If $B\in\Gamma_k^+$, then $T_{k-1}(B)$ is positive definite.
 \item $\Gamma_k^+$ is convex.
\end{enumerate}

\subsection{$\sigma_k$-curvatures on manifolds with boundary}
\label{subsec:bg/mfd}

Let $(X^{d},g)$ be a Riemannian manifold.  The \emph{Schouten tensor} is the section of $S^2T^\ast X$ given by
\[ P := \frac{1}{d-2}\left(\Ric - Jg\right), \]
where $\Ric$ is the Ricci tensor of $g$ and $J=\tr_g P$.  Thus $J=\frac{R}{2(d-1)}$ for $R$ the scalar curvature of $g$.  We write $g^{-1}P$ as the section of $\End(TX)$ obtained by raising an index of the Schouten tensor using the metric $g$.  Given $k\in\bN$, we define the \emph{$\sigma_k$-curvature} $\sigma_k^g$ of $(X^{d},g)$ by
\[ \sigma_k^g := \sigma_k(g^{-1}P) . \]
More precisely, $\sigma_k^g\in C^\infty(X)$ is the function which assigns to each point $x\in X$ the $k$-th elementary symmetric function of the matrix $g_x^{-1}P_x\in\Sym_{d}$ obtained by evaluating the section $g^{-1}P$ at $x$ to produce a linear map on $T_xX\cong\bR^{d}$.  Similarly, the \emph{$k$-th Newton tensor} $T_k^g$ of $(X^d,g)$ is the section of $S^2T^\ast X$ determined by
\[ g^{-1}T_k^g = T_k(g^{-1}P) . \]
When a Riemannian manifold $(X^d,g)$ is clearly specified by context, we omit $g$ from our notation and simply write $\sigma_k$ and $T_k$ for the $\sigma_k$-curvature and $k$-th Newton tensor, respectively, of $(X,g)$.

Suppose now that $(X^{n+1},g)$ is a Riemannian manifold with non-empty boundary $M^n=\partial X$.  We set $\og:=i^\ast g$, where $i\colon M\to X$ is the inclusion map, and denote by $\eta_g$ the outward-pointing unit normal vector field with respect to $g$ along $M$.  The \emph{second fundamental form} $A$ is the section of $S^2T^\ast M$ defined by
\[ A(Y,Z) = g\left(\nabla_Y\eta, Z\right) \]
for all $Y,Z\in T_xM$.  In this setting, we always denote by $N:=n+1$ the dimension of $X$ and use bars to denote geometric invariant determined by the restriction of $g$ to the boundary; e.g.\ $\oP$ denotes the Schouten tensor of $\og$.

Given $k\in\bN$, we define the \emph{$H_k$-curvature of $M$} by
\begin{equation}
 \label{eqn:defn_Hk}
 H_k^g := \sum_{j=0}^{k-1} \frac{(2k-j-1)!(N-2k+j)!}{(N-k)!(2k-2j-1)!!j!}\sigma_{2k-j-1,j}\left(\og^{-1}i^\ast P,\og^{-1}A\right),
\end{equation}
where $\og^{-1}i^\ast P,\og^{-1}A$ are sections of $\End(TM)$ obtained by raising an index using the metric $\og$ on $M$, and $H_k^g\in C^\infty(M)$ is evaluated at points on $p\in M$ by evaluating the sections $\og^{-1}i^\ast P$ and $\og^{-1}A$ at $p\in M$ to produce a linear map on $T_pM\cong\bR^n$.  When a Riemannian manifold $(X^{n+1},g)$ with non-empty boundary is clearly specified by context, we omit $g$ from our notation and simply write $H_k$ for the $H_k$-curvature.  In this case, we also write $\sigma_{k,\ell}$ for $\sigma_{k,\ell}\left(\og^{-1}i^\ast P,\og^{-1}A\right)$ and write $T_{k,\ell}$ for $T_{k,\ell}\left(\og^{-1}i^\ast P,\og^{-1}A\right)$.

\begin{remark}
 Our definition~\eqref{eqn:defn_Hk} of the $H_k$-curvature is originally due to S.\ Chen~\cite{Chen2009s}.  We draw heavily from~\cite{Chen2009s} for facts involving the $H_k$-curvature, though with two differences of convention: In our paper, we use $n$ to denote the dimension of the boundary $M$, and we use $\eta$ to denote the outward-pointing unit normal vector field.  By comparison, Chen uses $n$ to denote the dimension of the interior $X$, and also uses $n$ to denote the inward-pointing unit normal.
\end{remark}

In the special case $k=1$, we compute that $H_1=\frac{1}{n}\tr A=:H$ is the \emph{mean curvature} $H$ of $M$.  When studying the $H_k$-curvature for larger values of $k$, the consequence
\begin{equation}
 \label{eqn:gauss_codazzi}
 i^\ast P = \oP - HA_0 - \frac{1}{2}H^2\og + \frac{1}{n-2}\left(W(\eta,\cdot,\eta,\cdot) + A_0^2 - \frac{1}{2(n-1)}\lv A_0\rv^2\og\right)
\end{equation}
of the Gauss--Codazzi equations will be useful (see~\cite[Lemma~2.1]{Case2015b}), where $A_0:=A-H\og$ is the trace-free part of the second fundamental form and $W$ is the Weyl tensor of $g$.  For example, up to adding multiples of the pointwise conformally invariant scalars $\tr A_0^3$ and $\lp W(\eta,\cdot,\eta,\cdot),A_0\rp$, it holds that
\[ H_2 = H\oJ - \frac{n}{6}H^3 - \frac{1}{n-1}\lp A_0,\oP\rp + \frac{1}{2(n-1)}H\lv A_0\rv^2 . \]

As shown by S.\ Chen~\cite{Chen2009s}, a key property of the $H_k$-curvature is that it acts as a boundary term for the $\sigma_k$-curvature, in that, under suitable hypotheses on the underlying conformal manifold, the $L^2$-gradient of the functional $g\mapsto\int\sigma_k+\oint H_k$ on a conformal class is given exactly by the pair $(\sigma_k,H_k)$.  This is a consequence of the following formulae for the conformal linearizations of $\sigma_k$ and $H_k$.

\begin{lem}
 \label{lem:conformal_variation}
 Let $(X^{n+1},g)$ be a Riemannian manifold with non-empty boundary $M=\partial X$ and let $k\in\bN$.  If $k\geq 3$, assume additionally that $g$ is locally conformally flat.  For any $\Upsilon\in C^\infty(X)$, it holds that
 \begin{align}
  \label{eqn:genl_trans_sigmak} \left.\frac{\partial}{\partial t}\right|_{t=0} \sigma_k^{e^{2t\Upsilon}g} & = -2k\Upsilon\sigma_k - \delta\left(T_{k-1}(\nabla\Upsilon)\right) , \\
  \label{eqn:genl_trans_Hk} \left.\frac{\partial}{\partial t}\right|_{t=0} H_k^{e^{2t\Upsilon}g} & = -(2k-1)\Upsilon H_k + T_{k-1}(\eta,\nabla\Upsilon) - \odelta\left(S_{k-1}(\onabla\Upsilon)\right),
 \end{align}
 where $\delta=\tr_g\nabla$ and $\odelta=\tr_{\og}\onabla$ denote the divergence on $(X,g)$ and $(M,\og)$, respectively, and
 \[ S_{k-1} := \sum_{i=0}^{k-2} \frac{(2k-i-3)!(n-2k+i+2)!}{i!(n+1-k)!(2k-2i-3)!!}T_{2k-i-3,i} . \]
\end{lem}

\begin{proof}
 The expression~\eqref{eqn:genl_trans_sigmak} is well-known, and follows immediately from the formula
 \[ \left.\frac{\partial}{\partial t}\right|_{t=0} e^{2kt\Upsilon}\sigma_k^{e^{2t\Upsilon}g} = -\left\lp T_{k-1},\nabla^2\Upsilon\right\rp \]
 and the fact that $T_{k-1}$ is divergence-free under our hypotheses.

 To derive~\eqref{eqn:genl_trans_Hk}, recall from~\cite{Chen2009s} that
 \begin{equation}
  \label{eqn:divergence_Tqr}
  \odelta T_{q,r} = -\frac{(n-q)(q-r)}{q}T_{q-1,r}\left(P(\eta),\cdot\right) .
 \end{equation}
 where $P(\eta)$ denotes the vector field $(g^{-1}P)(\eta)$.  Fix $\Upsilon\in C^\infty(X)$ and, for brevity, denote
 \[ \sigma_{2k-i-1,i}^\bullet := \left.\frac{\partial}{\partial t}\right|_{t=0}e^{(2k-1)t\Upsilon}\sigma_{2k-i-1,i}^{e^{2t\Upsilon}g} . \]
 It follows immediately from the conformal transformation formulae for the Schouten tensor and the second fundamental form that
 \begin{align*}
  \sigma_{2k-i-1,i}^\bullet & = -\frac{i}{2k-i-1}\left\lp T_{2k-i-2,i-1}, \onabla^2\Upsilon + (\eta\Upsilon)A\right\rp \\
   & \quad + \frac{2k-2i-1}{2k-i-1}(\eta\Upsilon) \tr_{\og} T_{2k-i-2,i}
 \end{align*}
 for any $k\in\bN$ and any $i\in\bN_0$.  Using~\eqref{eqn:divergence_Tqr}, we compute that
 \begin{align*}
  \sigma_{2k-i-1,i}^\bullet & = -\frac{i}{2k-i-1}\odelta\left(T_{2k-i-2,i-1}(\onabla\Upsilon)\right) + \frac{i(i-1)}{2k-i-1}T_{2k-i-2,i-2}\left(P(\eta),\onabla\Upsilon\right) \\
   & \quad - \frac{i(n-2k+i+2)(2k-2i-1)}{(2k-i-1)(2k-i-2)}T_{2k-i-3,i-1}\left(P(\eta),\onabla\Upsilon\right) \\
   & \quad - i(\eta\Upsilon)\sigma_{2k-i-1,i-1} + \frac{(n-2k+i+2)(2k-2i-1)}{2k-i-1}(\eta\Upsilon)\sigma_{2k-i-2,i} .
 \end{align*}
 Inserting this into the definition of $H_k$ yields~\eqref{eqn:genl_trans_Hk}
\end{proof}

As a corollary of Lemma~\ref{lem:conformal_variation}, we find a functional with conformal gradient equal to $(\sigma_k,H_k)$.  In the case $N\not=2k$, this recovers the original observation of Chen~\cite{Chen2009s}, while when $N=2k$ this generalizes the functional found by Brendle and Viaclovsky~\cite{BrendleViaclovsky2004} to manifolds with boundary.

\begin{prop}
 \label{prop:boundary_functional}
 Let $(X^{n+1},g_0)$ be a compact Riemannian manifold with non-empty boundary $M=\partial X$ and let $k\in\bN$.  If $k\geq 3$, assume additionally that $g$ is locally conformally flat.  Let $\kC=\left\{ e^{2u}g_0\suchthat u\in C^\infty(X)\right\}$ denote the set of all metrics on $X$ which are pointwise conformally equivalent to $g_0$.  Define $\mS_k\colon\kC\to\bR$ by
 \begin{align*}
  \mS_k(g) & = \frac{1}{N-2k}\left[\int_X \sigma_k^g\,\dvol_g + \oint_M H_k^g\,\dvol_{\og} \right], && \text{if $N\not=2k$}, \\
  \mS_k(g) & = \int_0^1\left\{ \int_X u\,\sigma_k^{g_s}\,\dvol_{g_s} + \oint_M u\,H_k^{g_s}\,\dvol_{\og_s} \right\} ds, && \text{if $N=2k$},
 \end{align*}
 where $g=e^{2u}g_0$ and $g_s=e^{2su}g_0$ in the second case.  Then
 \begin{equation}
  \label{eqn:primitive}
  \left.\frac{d}{dt}\right|_{t=0}\mS_k(e^{2t\Upsilon}g) = \int_X \Upsilon\sigma_k\,\dvol_g + \oint_M \Upsilon H_k\,\dvol_{\og}
 \end{equation}
 for all $g\in\kC$ and all $\Upsilon\in C^\infty(M)$.
\end{prop}

\begin{proof}
 Suppose first that $N\not=2k$.  It follows immediately from Lemma~\ref{lem:conformal_variation} and the definition of $\mS_k$ that
 \begin{multline*}
  \left.\frac{d}{dt}\right|_{t=0}\mS_k(e^{2t\Upsilon}g) = \int_X \Upsilon\sigma_k\,\dvol_g - \frac{1}{N-2k}\int_X \delta\left(T_{k-1}(\nabla\Upsilon)\right)\dvol_g \\ + \oint_M \Upsilon H_k\,\dvol_{\og} - \frac{1}{N-2k}\oint_M \left(\odelta\left(S_{k-1}(\onabla\Upsilon)\right) - T_{k-1}(\eta,\nabla\Upsilon)\right) \dvol_{\og} .
 \end{multline*}
 The divergence theorem immediately yields~\eqref{eqn:primitive}.

 Suppose next that $N=2k$.  It follows immediately from Lemma~\ref{lem:conformal_variation} and the definition of $\mS_k$ that
 \begin{multline*}
  \left.\frac{d}{dt}\right|_{t=0}\mS_k(e^{2t\Upsilon}g) = \int_0^1\Biggl\{ \int_X \left[ \Upsilon\sigma_k^{g_s} - su\delta^{g_s}\left(T_{k-1}^{g_s}(\nabla^{g_s}\Upsilon)\right) \right] \dvol_{g_s} \\
   + \oint_M \left[ \Upsilon H_k^{g_s} + suT_{k-1}^{g_s}(\eta^{g_s},\nabla^{g_s}\Upsilon) - su\odelta^{\og_s}\left(S_{k-1}^{g_s}(\onabla^{\og_s}\Upsilon)\right) \right] \dvol_{\og_s} \Biggl\} ds .
 \end{multline*}
 Since both $T_{k-1}$ and $S_{k-1}$ are symmetric tensors, integrating by parts yields
 \begin{multline*}
  \left.\frac{d}{dt}\right|_{t=0}\mS_k(e^{2t\Upsilon}g) = \int_0^1\Biggl\{ \int_X \left[ \Upsilon\sigma_k^{g_s} - s\Upsilon\delta^{g_s}\left(T_{k-1}^{g_s}(\nabla^{g_s}u)\right) \right] \dvol_{g_s} \\
   + \oint_M \left[ \Upsilon H_k^{g_s} + s\Upsilon T_{k-1}^{g_s}(\eta^{g_s},\nabla^{g_s}u) - s\Upsilon\odelta^{\og_s}\left(S_{k-1}^{g_s}(\onabla^{\og_s}u)\right) \right] \dvol_{\og_s} \Biggl\} ds .
 \end{multline*}
 Using Lemma~\ref{lem:conformal_variation} again, we conclude that
 \begin{align*}
  \left.\frac{d}{dt}\right|_{t=0}\mS_k(e^{2t\Upsilon}g) & = \int_0^1\Biggl\{ \int_X \left[ \Upsilon\sigma_k^{g_s}\dvol_{g_s} + s\Upsilon\frac{\partial}{\partial s}\left(\sigma_k^{g_s}\dvol_{g_s}\right) \right] \\
   & \quad + \oint_M \left[ \Upsilon H_k^{g_s} + s\Upsilon\frac{\partial}{\partial s}\left(H_k^{g_s}\dvol_{\og_s}\right) \right] \Biggr\} ds .
 \end{align*}
 Changing the order of integration readily yields~\eqref{eqn:primitive}.
\end{proof}

\section{The boundary operators}
\label{sec:boundary}

The primary goal of this article is to study critical points of the functional $\mS_k$ from Proposition~\ref{prop:boundary_functional} for various classes of functions.  We are especially interested in extremizing $\mS_k$ and understanding the geometric significance of the extremal functions.  The goal of this section is to introduce multilinear conformally covariant operators that are well-adapted to this study, analogous to the role of the conformal Laplacian and the conformal Robin operator in the study of the Yamabe Problem on manifolds with boundary~\cite{Escobar1988,Escobar1992a}.  For the remainder of this article, we restrict our attention to noncritical dimensions $N\not=2k$.

Let $(X^{n+1},g)$ be a Riemannian manifold with non-empty boundary $M$ and fix a positive integer $k$ such that $N\not=2k$.  Given a positive function $u\in C^\infty(X)$, set
\begin{align}
 \label{eqn:Lk_defn} L_k(u) & = \left(\frac{N-2k}{2k}u\right)^{2k-1}u^{\frac{4k^2}{N-2k}}\sigma_k^{g_u}, \\
 \label{eqn:Bk_defn} B_k(u) & = \left(\frac{N-2k}{2k}u\right)^{2k-1}u^{\frac{2k(2k-1)}{N-2k}}H_k^{g_u} ,
\end{align}
where $g_u:=u^{\frac{4k}{N-2k}}g$.  Our goal is to show that $L_k$ and $B_k$ can be polarized to obtain multilinear conformally covariant operators.  This is a consequence of the basic symmetry and transformation properties of these operators.  The case of $L_k$ is as follows:

\begin{lem}
 \label{lem:Lk}
 Let $(X^{n+1},g)$ be a Riemannian manifold and let $k\in\bN$ be such that $n+1\not=2k$.  Define $L_k\colon C^\infty(X)\to C^\infty(X)$ by~\eqref{eqn:Lk_defn}.  Then
 \begin{align}
  \label{eqn:Lk_to_sigmak} L_k(1) & = \left(\frac{N-2k}{2k}\right)^{2k-1}\sigma_k, \\
  \label{eqn:Lk_homogeneous} L_k(cu) & = c^{2k-1}L_k(u), \\
  \label{eqn:Lk_transformation} L_k^{\hg}(u) & = e^{-\frac{(2k-1)N+2k}{2k}\Upsilon} L_k^g\left(e^{\frac{N-2k}{2k}\Upsilon}u\right)
 \end{align}
 for all $c\in\bR$ and all $u,\Upsilon\in C^\infty(X)$, where $L_k^{\hg}$ is defined with respect to $\hg:=e^{2\Upsilon}g$.  Moreover, $L_k(u)$ is a homogeneous polynomial of degree $2k-1$ in the two-jet of $u$.  In particular, the polarization of $L_k$ is conformally covariant and $(2k-1)$-linear.
\end{lem}

\begin{proof}
 Using the function $u=1$ in~\eqref{eqn:Lk_defn} immediately yields~\eqref{eqn:Lk_to_sigmak}.  For the remaining properties, recall that, as a section of $S^2T^\ast X$,
 \begin{multline}
  \label{eqn:schouten_transformation}
  P^{g_u} = P - \frac{2k}{N-2k}u^{-1}\nabla^2u \\ + \frac{2Nk}{(N-2k)^2}u^{-2}du\otimes du - \frac{1}{2}\left(\frac{2k}{N-2k}\right)^2u^{-2}\lv\nabla u\rv^2g .
 \end{multline}
 where all quantities on the right-hand side are computed with respect to $g$.  In particular, $P^{g_{cu}}=P^{g_u}$, so that~\eqref{eqn:Lk_homogeneous} follows immediately from~\eqref{eqn:Lk_defn}.

 Next, observe that our definition~\eqref{eqn:Lk_defn} is equivalently written
 \begin{equation}
  \label{eqn:Lk_defn2}
  L_k(u) = \frac{2ku^{-1}}{N-2k}\sigma_k\left(g^{-1}B\right) ,
 \end{equation}
 where we use~\eqref{eqn:schouten_transformation} to realize $B$ as
 \begin{equation}
  \label{eqn:B_defn}
   B = \left(\frac{N-2k}{2k}u\right)^2P - \frac{N-2k}{2k}u\nabla^2u + \frac{N}{2k}du \otimes du - \frac{1}{2}\lv\nabla u\rv^2 g .
 \end{equation}
 It is readily computed from~\eqref{eqn:sym_sigmak} that $\sigma_j(du\otimes\nabla u)=0$ for $j\geq 2$.  Therefore
 \[ \sigma_k\left(g^{-1}(N\,du\otimes du - k\lv\nabla u\rv^2 g)\right) = 0 . \]
 Combining this with~\eqref{eqn:B_defn}, we see that $\sigma_k(g^{-1}B)$ is a homogeneous polynomial of degree $k$ in $u^2,du\otimes du,u\nabla^2u$ for which the coefficient of $(du\otimes du)^k$ vanishes identically.  In particular, we may divide $\sigma_k(g^{-1}B)$ by $u$ to obtain a homogeneous polynomial of degree $2k-1$ in $u,du,\nabla^2u$.  Inserting this into~\eqref{eqn:Lk_defn2} implies that $L_k(u)$ is a homogeneous polynomial of degree $2k-1$ in the two-jet of $u$.  We may thus polarize $L_k$ to obtain a $\bR$-multilinear operator
 \begin{equation}
  \label{eqn:Lk_polarization}
  L_k \colon \left(C^\infty(X)\right)^{2k-1} \to C^\infty(X) .
 \end{equation}

 Finally, let $\hg=e^{2\Upsilon}g$ be a pointwise conformal rescaling of $g$.  In terms of $\hg$ we have that
 \begin{equation}
  \label{eqn:choice_of_weight}
  g_u = \left(e^{-\frac{N-2k}{2k}\Upsilon}u\right)^{\frac{4k}{N-2k}}\hg .
 \end{equation}
 Inserting~\eqref{eqn:choice_of_weight} into~\eqref{eqn:Lk_defn} yields~\eqref{eqn:Lk_transformation}.  Writing~\eqref{eqn:Lk_transformation} in terms of the polarization~\eqref{eqn:Lk_polarization} of $L_k$ yields the conformal covariance of the polarization of $L_k$.
\end{proof}

The fact that $B_k$ as in~\eqref{eqn:Bk_defn} can be polarized to obtain a multilinear conformally covariant operator is proven in a similar manner.

\begin{lem}
 \label{lem:Bk}
 Let $(X^{n+1},g)$ be a Riemannian manifold with nonempty boundary $M^n=\partial X$ and define $B_k\colon C^\infty(X)\to C^\infty(M)$ by~\eqref{eqn:Bk_defn}.  Then
 \begin{align}
  \label{eqn:Bk_to_sigmak} B_k(1) & = \left(\frac{N-2k}{2k}\right)^{2k-1}H_k, \\
  \label{eqn:Bk_homogeneous} B_k(cu) & = c^{2k-1}B_k(u), \\
  \label{eqn:Bk_transformation} B_k^{\hg}(u) & = e^{-\frac{(2k-1)N}{2k}\Upsilon} B_k^g\left(e^{\frac{N-2k}{2k}\Upsilon}u\right)
 \end{align}
 for all $c\in\bR$ and all $u,\Upsilon\in C^\infty(X)$, where $B_k^{\hg}$ is defined with respect to $\hg:=e^{2\Upsilon}g$.  Moreover, $B_k(u)$ is a homogeneous polynomial of degree $2k-1$ in the two-jet of $u$.  In particular, the polarization of $B_k$ is conformally covariant and $(2k-1)$-linear.
 Finally, $B_k(u)$ depends only on $u,\onabla u,\onabla^2u,\eta u$; i.e.\ $B_k$ is a fully nonlinear second order operator in $u$, but it does not depend on second order derivatives in directions transverse to $M$.
\end{lem}

\begin{proof}
 Using the function $u=1$ in~\eqref{eqn:Bk_defn} immediately yields~\eqref{eqn:Bk_to_sigmak}.

 The fact that $B_k(u)$ depends at a point only on $u,\onabla u,\onabla^2u,\eta u$ at that point follows from three observations.  First, we recall that each of $A_0$, $W$, $\eta$, and $\og$ is conformally covariant.  Second, \eqref{eqn:schouten_transformation} gives the transformation formula for $\oP^{\og_u}$ entirely in terms of tangential data.  Third, the mean curvature transforms conformally by
 \[ u^{\frac{2k}{N-2k}}H^{g_u} = H + \frac{2k}{N-2k}u^{-1}\eta u . \]
 Inserting these observations into~\eqref{eqn:gauss_codazzi} yields the claimed dependence of $B_k(u)$ at a point only on $u$, $\onabla u$, $\onabla^2u$, and $\eta u$.

 For the remaining properties, recall that, as a section of $S^2T^\ast M$, it holds that
 \begin{equation}
  \label{eqn:second_fundamental_form_transformation}
  u^{-\frac{2k}{N-2k}}A^{g_u} = A_g + \frac{2k}{N-2k}u^{-1}\eta u\,\og .
 \end{equation}
 In particular, $A^{g_{cu}}=c^{\frac{2k}{N-2k}}A^{g_u}$, so that~\eqref{eqn:Bk_homogeneous} follows immediately from~\eqref{eqn:Bk_defn} and the homogeneity of the Schouten tensor.  Using\eqref{eqn:Bk_defn}, \eqref{eqn:schouten_transformation} and~\eqref{eqn:second_fundamental_form_transformation}, we see that
 \begin{equation}
  \label{eqn:Bk_defn2}
  B_k(u) = \sum_{j=0}^{k-1} \frac{(2k-j-1)!(N-2k+j)!}{(N-k)!(2k-2j-1)!!j!} \sigma_{2k-j-1,j}\left( \og^{-1}i^\ast B, \og^{-1}C \right),
 \end{equation}
 where $B$ is as in~\eqref{eqn:B_defn} and
 \[ C := n\eta u + \frac{N-2k}{2k}Hu . \]
 It follows immediately from~\eqref{eqn:Bk_defn2} that $B_k$ is a homogeneous polynomial of degree $2k-1$ in the two-jet of $u$.  In particular, the polarization
 \begin{equation}
  \label{eqn:Bk_polarization}
  B_k \colon \left(C^\infty(X)\right)^{2k-1} \to C^\infty(M)
 \end{equation}
 of $B_k$ is $\bR$-multilinear.  Finally, inserting~\eqref{eqn:choice_of_weight} into~\eqref{eqn:Bk_defn} immediately yields~\eqref{eqn:Bk_transformation}.  Writing~\eqref{eqn:Bk_transformation} in terms of the polarization~\eqref{eqn:Bk_polarization} of $B_k$ yields the conformal covariance of the polarization of $B_k$.
\end{proof}

The main result of this section is the fact that $(L_k,B_k)$ is a formally self-adjoint conformally covariant system.

\begin{thm}
 \label{thm:k_boundary_operators}
 Let $(X^{n+1},g)$ be a compact Riemannian manifold with boundary $M^n=\partial X$.  Fix $k\in\bN$ with $N\neq 2k$; if $k\geq3$, assume further that $g$ is locally conformally flat.  Define $\mQ_k\colon\left(C^\infty(X)\right)^{2k}\to\bR$ by
 \[ \mQ_k(u_1,\dotsc,u_{2k}) = \int_X u_1\,L_k(u_2,\dotsc,u_{2k})\dvol_g + \oint_M u_1\,B_k(u_2,\dotsc,u_{2k})\dvol_{\og} . \]
 Then $\mQ_k$ is a symmetric conformally covariant $\bR$-multilinear form; indeed,
 \[ \mQ_k^{\hg}(u_1,\dotsc,u_{2k}) = \mQ_k^g\left(e^{\frac{N-2k}{2k}\Upsilon}u_1,\dotsc,e^{\frac{N-2k}{2k}\Upsilon}u_{2k}\right) \]
 for all $u_1,\dotsc,u_{2k}\in C^\infty(X)$ and all metrics $\hg=e^{2\Upsilon}g$, $\Upsilon\in C^\infty(X)$.
\end{thm}

\begin{proof}
 \eqref{eqn:Lk_transformation} and~\eqref{eqn:Bk_transformation} imply that the integrals
 \[ \int_X u_1\,L_k(u_2,\dotsc,u_{2k})\dvol_g \quad\text{and}\quad \oint_M u_1\,B_k(u_2,\dotsc,u_{2k})\dvol_{\og} , \]
 respectively, are conformally covariant.  Therefore $\mQ_k$ is conformally covariant.

 To see that $\mQ_k$ is symmetric, first observe that~\eqref{eqn:Lk_to_sigmak} and~\eqref{eqn:Bk_to_sigmak} imply that
 \begin{equation}
  \label{eqn:mQk_to_sigmak}
  \mQ_k(1,\dotsc,1) = \left(\frac{N-2k}{2k}\right)^{2k-1}\left[ \int_X \sigma_k\,\dvol_g + \oint_X H_k\,\dvol_{\og} \right] .
 \end{equation}
 Let $\Upsilon\in C^\infty(X)$ and consider the one-parameter family $g_t:=e^{2t\Upsilon}g$ of pointwise conformally rescaled metrics.  Lemma~\ref{lem:Lk} and Lemma~\ref{lem:Bk} imply
 \begin{align*}
  \left.\frac{\partial}{\partial t}\right|_{t=0}\sigma_k&(g_t)= \\ 
   \left(\frac{N-2k}{2k}\right)^{1-2k}&\left[ -\frac{(2k-1)N+2k}{2k}\Upsilon L_k(1,\dotsc,1) + \frac{(2k-1)(N-2k)}{2k}L_k(\Upsilon,1,\dotsc,1)\right], \\
  \left.\frac{\partial}{\partial t}\right|_{t=0}H_k(g_t) & =\\
   \left(\frac{N-2k}{2k}\right)^{1-2k}&\left[ -\frac{(2k-1)N}{2k}\Upsilon B_k(1,\dotsc,1) + \frac{(2k-1)(N-2k)}{2k}B_k(\Upsilon,1,\dotsc,1)\right],
 \end{align*}
 respectively.  In particular, we obtain that
 \begin{multline*}
  \left(\frac{N-2k}{2k}\right)^{2k-1}\left.\frac{d}{dt}\right|_{t=0}\left(\int_X\sigma_k^{g_t}\,\dvol_{g_t} + \oint_X H_k^{g_t}\,\dvol_{\og_t}\right) \\
  \begin{aligned}
   & = \int_X \left[ \frac{N-2k}{2k}\Upsilon L_k(1,\dotsc,1) + \frac{(2k-1)(N-2k)}{2k}L_k(\Upsilon,1,\dotsc,1) \right]\dvol_g \\
   & \quad + \oint_M \left[ \frac{N-2k}{2k}\Upsilon B_k(1,\dotsc,1) + \frac{(2k-1)(N-2k)}{2k}B_k(\Upsilon,1\dotsc,1) \right]\dvol_{\og} .
  \end{aligned}
 \end{multline*}
 Comparing this to~\eqref{eqn:primitive} yields
 \begin{equation}
  \label{eqn:mQk_symmetric1}
  \mQ_k(\Upsilon,1,\dotsc,1) = \mQ_k(1,\Upsilon,1,\dotsc,1) .
 \end{equation}
 Since $\mQ_k$ is conformally invariant, this holds for all metrics in the conformal class $\kC$ of $g$.

 Now let $u,\Upsilon\in C^\infty(X)$ and again consider the one-parameter family $g_t:=e^{2t\Upsilon}g$ of conformally equivalent metrics.  Define $F\colon\bR\to\bR$ by
 \begin{multline*}
  F(t) = \int_X \left[ L_k^{g_t}(u,1,\dotsc,1)-uL_k^{g_t}(1,\dotsc,1) \right]\dvol_{g_t} \\ + \oint_X \left[ B_k^{g_t}(u,1,\dotsc,1) - u\rv_M B_k^{g_t}(1,\dotsc,1) \right]\dvol_{\og_t} .
 \end{multline*}
 We conclude from~\eqref{eqn:mQk_symmetric1} that $F^\prime(0)=0$.  On the other hand, using~\eqref{eqn:Lk_transformation} and~\eqref{eqn:Bk_transformation} to compute $F^\prime(0)$ as above yields
 \begin{align*}
  F^\prime(0) & = \int_X \Biggl[ \frac{N-2k}{2k}\Upsilon L_k(u,1,\dotsc,1) + \frac{(2k-2)(N-2k)}{2k}L_k(u,\Upsilon,1,\dotsc,1) \\
   & \qquad\qquad - \frac{(2k-1)(N-2k)}{2k}uL_k(\Upsilon,1,\dotsc,1) \\
   & \qquad\qquad + \frac{N-2k}{2k}L_k(\Upsilon u,1,\dotsc,1) - \frac{N-2k}{2k}\Upsilon u L_k(1,\dotsc,1) \Biggr]\dvol_g \\
  & \quad + \oint_M \Biggl[ \frac{N-2k}{2k}\Upsilon\rv_M B_k(u,1,\dotsc,1) + \frac{(2k-2)(N-2k)}{2k}B_k(u,\Upsilon,1,\dotsc,1) \\
   & \qquad\qquad - \frac{(2k-1)(N-2k)}{2k}u\rv_M L_k(\Upsilon,1,\dotsc,1) \\
   & \qquad\qquad + \frac{N-2k}{2k}B_k(\Upsilon u,1,\dotsc,1) - \frac{N-2k}{2k}(\Upsilon u)\rv_M L_k(1,\dotsc,1) \Biggr]\dvol_{\og} .
 \end{align*}
 Using~\eqref{eqn:mQk_symmetric1} to eliminate the third and sixth lines of this computation and then symmetrizing implies that
 \[ \mQ_k(\Upsilon_1,\Upsilon_2,1,\dotsc,1) = \mQ_k(\Upsilon_2,\Upsilon_1,1,\dotsc,1) = \mQ_k(1,\Upsilon_1,\Upsilon_2,1,\dotsc,1) \]
 for all $\Upsilon_1,\Upsilon_2\in C^\infty(X)$.  Continuing in this manner, we deduce that $\mQ_k$ is symmetric.
\end{proof}

For our purposes, the most important consequence of Theorem~\ref{thm:k_boundary_operators} is the following corollary used to prove the convexity of $\mS_k$ in the positive $k$-cone when the boundary value $i^\ast g$ is fixed.

\begin{cor}
 \label{cor:nonnegative2}
 Let $(X^{n+1},g)$ be a compact Riemannian manifold with boundary $M^n=\partial X$.  Fix $k\in\bN$ with $N\neq 2k$; if $k\geq3$, assume further that $g$ is locally conformally flat.  For any $v\in C^\infty(X)$ with $v\rv_M=0$ and any positive $u\in C^\infty(X)$ with $u^{\frac{4k}{N-2k}}g\in\overline{\Gamma_k^+}$, it holds that
 \[ \mQ_k(v,v,u,\dotsc,u) \geq 0 . \]
\end{cor}

\begin{proof}
 Theorem~\ref{thm:k_boundary_operators} implies that
 \[ \mQ_k^g(v,v,u,\dotsc,u) = \mQ_k^{g_u}(u^{-1}v,u^{-1}v,1,\dotsc,1), \]
 where $g_u=u^{\frac{4k}{N-2k}}g$.  Since $v\rv_M=0$, we conclude that
 \begin{equation}
  \label{eqn:n21}
  \mQ_k^g(v,v,u,\dotsc,u) = \int_X u^{-1}v L_k^{g_u}(u^{-1}v,1,\dotsc,1)\,\dvol_{g_u} .
 \end{equation}
 On the other hand, Lemma~\ref{lem:conformal_variation} and Lemma~\ref{lem:Lk} imply that
 \begin{multline}
  \label{eqn:n22}
  (2k-1)\left(\frac{N-2k}{2k}\right)^{2-2k}L_k(\Upsilon,1,\dotsc,1) \\ = -\delta\left(T_{k-1}(\nabla\Upsilon)\right) + \frac{(N-2k)(2k-1)}{2k}\sigma_k\Upsilon .
 \end{multline}
 Since $g_u\in\overline{\Gamma_k^+}$, it holds that $\sigma_k^{g_u}\geq0$ and $T_{k-1}^{g_u}\geq0$.  Combining~\eqref{eqn:n21} and~\eqref{eqn:n22} then yields the result.
\end{proof}

\section{Inequality}\label{sec:inequality}

Consider $g_u:=u^{\frac{4k}{N-2k}}g \in [g]$, satisfying $g_{u}\in \Gamma_k^+(\overline{X})$,
and $u=f$ on $M=\partial X$.  Our first goal is to show that we can conformally deform $g_u$ to obtain a metric $\tilde g_u$ with $\sigma_k(\tilde g_u)=0$ and $\tilde g_u\rv_M=g_u\rv_M$.  To that end, recall the following existence result of Bo Guan~\cite{Guan2007} for the Dirichlet problem.

\begin{thm}
 \label{thm:Guan}
 Assume $\psi(x)\in C^\infty (\overline{X}), \psi>0$ and $h_{\alpha \beta}\in [g\rvert_{M}]\cap C^\infty (M)$.  There is a solution $g_u\in C^{1,1}(\oX)\cap \Gamma_k^+(\oX)$ of the boundary value problem
 \begin{equation}
  \begin{cases}
  \sigma_k^{g_u}=\psi, & \text{in $X$}, \\
  g_u=h, & \text{on $M$}, \\
 \end{cases}
 \end{equation}
 provided there exists a metric $g_{\mathrm{sub}}\in C^3(\overline{X})\cap \Gamma_k^+ (\overline{X})$ with
 $\sigma_k^{g_{\mathrm{sub}}}\geq \psi$ on $\overline{X}$, and $g_{\mathrm{sub}}|_{M}=h$. Moreover,
 \begin{equation}
  \label{est}
  \|g_u\|_{C^{1,1}}\leq  C(\|g_{\mathrm{sub}}\|_{C^{3}}, g, n, k) .
 \end{equation}
 Here the norm $\|\cdot \|_{C^{1,1}}$ is with respect to the background metric $g$.
\end{thm}

\begin{remark}
 While Theorem~\ref{thm:Guan} is not explicitly stated in \cite{Guan2007}, it is well-known to the experts and contained within the proof of~\cite[Theorem~1.1]{Guan2007}.
\end{remark}

We now construct the metric $\tilde g_u$ desired above.  Since $g_{u}\in \Gamma_k^+(\overline{X})$, there exists a $\delta>0$, such that
\[ \sigma_k^{g_u} \geq \delta>0 . \]
We can apply Theorem~\ref{thm:Guan} to find a family of solutions $u_j$, $j\in\bN$, to the problems
\[ \begin{cases}
    \sigma_k^{g_{u_j}}=\frac{\delta}{j}, & \text{in $X$}, \\
    u_j=f, & \text{on $M$} .
   \end{cases} \]
Next, using the a priori estimate~\eqref{est}, we deduce the uniform bound
\[ \lVert u_j\rVert_{C^{1,1}}\leq C (\lVert f\rVert_{C^3}, g, n, k). \]
Taking a subsequence if necessary, we conclude that $u_k\to u_\infty$ in $C^{1,1}$ and $u_\infty$ solves
\begin{equation}
 \label{eqn:infty}
 \begin{cases}
  \sigma_k^{g_{u_\infty}}=0, & \text{in $X$}, \\
  u_\infty=f, & \text{on $M$} , \\
 \end{cases}
\end{equation}
Thus $\tilde g_u:=g_{u_\infty}$ is as desired.

Our main goal in this section is to use the convexity of the functional $\mQ_k$ to show that the metric $\tilde g_u$ is the unique solution to~\eqref{eqn:infty}.  Indeed, this will establish the following fully nonlinear Sobolev trace inequality.

\begin{thm}
 \label{thm:inequality}
 Let $g_u\in[g]$ be such that $g_u\in\Gamma_k^+(\overline{X})$ and $u\rv_M=f$ on $M=\partial X$.  Then there exists a unique positive function $u_\infty\in C^{1,1}(\overline{X})$ of~\eqref{eqn:infty}.  Moreover, $u_\infty$ is the minimizer of $\mE_k(u):=\mQ_k(u,\dotsc,u)$ in the class
 \[ \mC_{k,f} := \left\{ u\in C^{1,1}(\overline{X}) \suchthat g_u\in\overline{\Gamma_k^+(\overline{X})}, u\rv_M=f \right\} ; \]
 that is,
 \[ \mE_k(u) \geq \mE_k(u_\infty) \]
 for all $u\in\mC_{k,f}$, with equality if and only if $u=u_\infty$.
\end{thm}

\begin{proof}
 The existence of a solution $u_\infty$ to \eqref{eqn:infty} has already been proven.  To prove it is the minimizer, we compute the first and second variations of the energy functional $\mE_k(u)$.  To that end, let $u\in C^\infty(\oX)$ be a positive function such that $g_u=u^{\frac{4k}{N-2k}}g\in\Gamma_k^+(\oX)$, and let $v\in C^\infty(\oX)$ be such that $v\rv_M=0$.  Since $\mQ_k$ is symmetric and $v\rv_M\equiv0$, we compute that
 \begin{equation}
  \begin{split}
  \label{eqn:first_variation}
  \left.\frac{d}{dt}\right|_{t=0}\mE_k(u+tv) & = 2k\mQ_k(v,u,\dotsc,u) \\
   & = 2k\int_X v L_k(u,\dotsc,u)\,\dvol_g,\\
  \end{split}
 \end{equation}
 Similarly, since $\mQ_k$ is symmetric and multilinear, we compute that
 \begin{equation}
  \label{eqn:second_variation}
  \left.\frac{d^2}{dt^2}\right|_{t=0}\mE_k(u+tv) = 2k(2k-1)\mQ_k(v,v,u,\dotsc,u) \geq 0 .
 \end{equation}
 where the inequality follows from Corollary~\ref{cor:nonnegative2}.
%


 Now, \eqref{eqn:first_variation} implies that the solution $u_\infty$ to~\eqref{eqn:infty} is a critical point of the functional $\mE_k\colon\mC_{f,k}\to\bR$.  By~\eqref{eqn:second_variation}, we see that $\mE_k\colon\mC_{f,k}\to\bR$ is a convex functional.  Since $\mC_{f,k}$ is convex, $u_\infty$ minimizes $\mE_k$ in $\mC_{f,k}$.  Indeed, if not, then there is a $u\in\mC_{f,k}$ such that $\mE_k(u)<\mE_k(u_\infty)$.  Since $\mC_{f,k}$ is convex, it holds that $tu+(1-t)u_\infty\in\mC_{f,k}$ for all $t\in[0,1]$.  Denote $\mE_k(t):=\mE_k(tu+(1-t)u_\infty)$.  Since $\mE_k(u)<\mE_k(u_\infty)$, there is a $t^\ast\in[0,1]$ such that $\mE_k^\prime(t^\ast)<0$.  This contradicts the facts that $\mE_k^\prime(0)=0$ and $\mE_k^{\prime\prime}(t)\geq0$ for all $t\in[0,1]$.

 To prove $u_\infty$ is also the unique minimizer. Suppose there are two minimizers $u$ and $v$.  Then $u$ and $v$ are both solutions to \eqref{eqn:infty}.  Let $f_0$ be the function so that $e^{2f_0}= u^{\frac{4k}{N-2k}}$, and $f_1$ be the function so that $e^{2f_1}= v^{\frac{4k}{N-2k}}$.
 Let $f_t= f_0+ t(f_1-f_0)$, $t\in [0,1]$, and $P_{t}= P^{e^{2f_t}g }$.  It is well known that
 \[ P_t= P^g - \nabla^2 f_t+ \nabla f_t \otimes \nabla f_t -\frac{1}{2} |\nabla f_t |^2 g. \]
 Consider the operator
 \begin{equation}
 L_{ij}:= \int_{0}^1 {T_{k-1}} (g^{-1}P_t)_{ij}\,dt .
 \end{equation}
 Since $P_t\in\overline{\Gamma_k^+}$, the operator $L$ is elliptic.  Then $f_1-f_0$ satisfies the elliptic equation
 \[ L_{ij}  (f_1-f_0)_{ij} +   \int_{0}^1 {T_{k-1}} (P_t)_{ij} \left[2 \nabla_i f_t \nabla_j (f_1-f_0)-g( \nabla_k f_t, \nabla_k (f_1-f_0)) g_{ij}\right]dt =0. \]
 The boundary condition in~\eqref{eqn:infty} implies that $f_1-f_0 =0$ on $\partial X$.  By the comparison principle, $f_1=f_0$ on $\overline{X}^{n+1}$.  Thus $u=v$ on $\overline{X}^{n+1}$. This gives the proof of uniqueness.
\end{proof}

\section{An Obata theorem}
\label{sec:obata}

The goal of this section is to give a partial classification of metrics $g\in\overline{\Gamma_k^+}$ in the conformal class of the round metric on the upper hemisphere $S_+^{n+1}$ for which $\sigma_k^g\equiv0$ and $H_k^g$ is constant along the boundary $S^n=\partial S_+^{n+1}$.  These results generalize Escobar's classification of scalar flat metrics on the upper hemisphere for which the boundary has constant mean curvature~\cite{Escobar1988}.  We begin by making a number of useful observations.  For possible future applications, we state these results in the most general geometric situation possible.

First, the assumption $g\in\overline{\Gamma_k^+}$ in the interior of a compact Riemannian manifold implies that the restriction of its Schouten tensor to the boundary is in the closure of the $(k-1)$-cone.  In particular, we have the following useful inequality:

\begin{lem}
 \label{lem:signs}
 Fix $k\in\bN$.  Let $(X^{n+1},g)$ be a compact Riemannian manifold with $g\in\overline{\Gamma_k^+}$.  Given a nonnegative integer $j\leq k-1$, it holds that
 \[ \left\lp \tf T_{j,j},P\rv_{TM}\right\rp \leq 0 \]
 with equality if and only if $\rank P\rv_{TM}\leq j-1$, where
 \[ \tf T_{j,j} := T_j(P\rv_{TM}) - \frac{n-j}{n}\sigma_j(P\rv_{TM})\og . \]
\end{lem}

\begin{proof}
 Since $g\in\overline{\Gamma_k^+}$, it follows that $T_j(\eta,\eta)\geq0$ for all nonnegative integers $j\leq k-1$, where $T_j$ is the $j$-th Newton tensor determined by the Schouten tensor $P$ of $g$ and $\eta$ is the outward-pointing unit normal along $M$.  By~\cite[Lemma~3]{Chen2009s},
 \[ T_j(\eta,\eta) = \sigma_j(P\rv_{TM}) . \]
 In particular, $\og^{-1}P\rv_{TM}\in\overline{\Gamma_{k-1}^+}$.  The conclusion now follows from the Newton inequalities (cf.\ \cite[Lemma~23]{Viaclovsky2000}).
\end{proof}

Second, the expression for $H_k$ simplifies in a useful way for compact Riemannian manifolds with umbilic boundary.

\begin{lem}
 \label{lem:simplify_Hk}
 Fix $k\in\bN$.  Let $(X^{n+1},g)$ be a compact Riemannian manifold with umbilic boundary $M=\partial X$; if $k\geq 2$, assume additionally that $g$ is locally conformally flat.  Then
 \begin{align*}
  H_k = \sum_{j=0}^{k-1} \frac{(n-j)!}{(n+1-k)!(2k-2j-1)!!}H^{2k-2j-1}\sigma_{j,j},
 \end{align*}
 where $\sigma_{j,j}:=\sigma_j(P\rv_{TM})$.
\end{lem}

\begin{proof}
 Given nonnegative integers $j\leq s$, define $\sigma_{s,j}:=\sigma_{s,j}(P\rv_{TM},A)$.  Since $M$ is umbilic, $A=H\og$.  In particular, we observe that
 \[ \sigma_{s,j} = \frac{(n-j)!j!}{(n-k)!k!}H^{k-j}\sigma_{j,j} . \]
 Inserting this into the definition of $H_k$ yields the result.
\end{proof}

Third, there is a useful divergence formula which is especially useful for studying metrics for which $H_k$ is constant.  Note that we do not explicitly record the divergence terms in this formula, though they are readily recovered from the proof.

\begin{lem}
 \label{lem:simplify_Tketa}
 Fix $k\in\bN$.  Let $(X^{n+1},g)$ be a compact Riemannian manifold with umbilic boundary $M=\partial X$; if $k\geq 2$, assume additionally that $g$ is locally conformally flat.  Then
 \[ T_k(\eta,\onabla u) + \frac{N-k}{N-1}H_k\oDelta u \cong -\sum_{j=0}^{k-1}\frac{(n-j-1)!}{(n-k)!(2k-2j-1)!!}H^{2k-2j-1}\left\lp \tf T_{j,j}, \onabla^2u \right\rp \]
 for all $u\in C^\infty(M)$, where $\cong$ means that equality holds modulo the image of $\odelta$.
\end{lem}

Fourth, the presence of conformal Killing fields on the round sphere gives rise to a useful integral identity for metrics with $\sigma_k$ constant.

\begin{lem}
 \label{lem:cKilling}
 Let $(S_+^{n+1},g_0)$ be the round upper hemisphere and fix $k\in\bN$.  Let $g=u^2g_0\in\overline{\Gamma_k^+}$ be a metric such that $\sigma_k(g)=0$.  Then
 \begin{equation}
  \label{eqn:cKilling}
  \oint_{S^n} u\sigma_{k,k} = 0 .
 \end{equation}
\end{lem}

\begin{proof}
 Consider $S_+^{n+1}$ as the set
 \[ S_+^{n+1} = \left\{ (x_0,\dotsc,x_{n+1}) \in \bR^{n+2} \suchthat \sum_{i=0}^{n+1}x_i^2=1, x_{n+1}\geq 0 \right\} \]
 with $g_0$ the metric induced by the Euclidean metric in $\bR^{n+2}$.  Denote $S^n=\partial S_+^{n+1}$.  Then the restriction of $X=-\nabla x_{n+1}$ to $S_+^{n+1}$ is a conformal Killing field and $X\rv_{S^n}=\eta_0$, the outward-pointing unit normal along $S^n$ with respect to $g_0$.  It follows that $X\rv_{S^n}=u\eta$, where $\eta$ is the outward-pointing unit normal along $S^n$ with respect to $g$.  Since $X$ is a conformal Killing field and $\sigma_k^g=0$, we have that $\lp\mL_Xg,T_k\rp=0$ for $\mL_Xg$ the Lie derivative of $g$ in the direction $X$.  Integrating this over $S_+^{n+1}$ and using the divergence theorem yields
 \[ 0 = \frac{1}{2}\int_{S_+^{n+1}} \left\lp \mL_Xg, T_k\right\rp = \oint_{S^n} T_k(\eta,X) = \oint_{S^n} u\sigma_{k,k} , \]
 as desired.
\end{proof}

We are now able to prove Theorem~\ref{thm:intro/obata}, which we restate here for the convenience of the reader

\begin{thm}
 \label{thm:obata}
 Let $\Gamma_k^+$ be the positive $k$-cone on the round $(n+1)$-dimensional hemisphere.  Suppose that $g\in\overline{\Gamma_k^+}$ is such that $\sigma_k=0$ and $H_k$ is a positive constant along $M=\partial S_+^{n+1}$.
 Suppose that
 \begin{equation}
  \label{eqn:obata_assumption}
  \sup_M H \leq (k+1)\inf_M H .
 \end{equation}
 Then $g$ is flat.
\end{thm}

\begin{remark}
 As in the proof of Lemma~\ref{lem:signs}, the assumptions $g\in\overline{\Gamma_k^+}$ and $H_k$ constant yield the estimate
 \[ H^{2k-1} \leq \frac{(n+1-k)!(2k-1)!!}{n!}H_k . \]
 for the supremum of $H$.
\end{remark}

\begin{proof}
 We present the proof for the case $k\geq2$; the case $k=1$ is handled by Escobar~\cite{Escobar1988}, and can be recovered from our argument with minor adjustments.

 Under our assumptions, $T_k$ is divergence-free.  Since $\sigma_k\equiv0$, it also holds that $T_k$ is trace-free.  On the other hand, since $u^{-2}g$ is Einstein, it holds that
 \begin{equation}
  \label{eqn:conformally_einstein}
  \tf(uP+\nabla^2u)=0 ,
 \end{equation}
 where $\tf$ denotes the trace-free part of its argument.  Therefore
 \begin{equation}
  \label{eqn:ibp}
  -(k+1)\int_{S_+^{n+1}} u\sigma_{k+1} = \int_{S_+^{n+1}} \left\lp T_k, \nabla^2u \right\rp = \oint_{S^n} T_k(\nabla u,\eta) .
 \end{equation}
 Since $g\in\overline{\Gamma_k^+}$, Newton's inequalities (cf.\ \cite[Lemma~23]{Viaclovsky2000}) imply that the left-hand side of~\eqref{eqn:ibp} is nonnegative with equality if and only if $P$ has rank at most $k-1$.  On the other hand, the splitting $\nabla u=\onabla u+(\eta u)\eta$ into its tangential and normal components along $S^n$ implies, when combined with Lemma~\ref{lem:simplify_Tketa}, that
 \begin{multline}
  \label{eqn:boundary_eval}
  \oint_{S^n} T_k(\nabla u,\eta) = \oint_{S^n} (\eta u)\sigma_{k,k} \\ - \sum_{j=0}^{k-1}\frac{(n-j-1)!}{(n-k)!(2k-2j-1)!!}\oint_{S^n}H^{2k-2j-1}\left\lp \tf T_{j,j},\onabla^2u \right\rp .
 \end{multline}
 The transformation formula for the mean curvature and the fact that $S^n$ is totally geodesic with respect to $g_0$ implies that $\eta u=Hu$.  Combining this observation with~\eqref{eqn:conformally_einstein}, \eqref{eqn:ibp}, and~\eqref{eqn:boundary_eval} yields
 \[ 0 \leq \oint_{S^n} Hu\sigma_{k,k} + \sum_{j=0}^{k-1}\frac{(n-j-1)!}{(n-k)!(2k-2j-1)!!}\oint_{S^n}H^{2k-2j-1}\left\lp \tf T_{j,j}, uP\right\rp . \]
 In particular, Lemma~\ref{lem:signs} implies that
 \begin{equation}
  \label{eqn:ibp2}
  0 \leq \oint_{S^n} Hu\sigma_{k,k} + \oint_{S^n} Hu\left(k\sigma_{k,k} - \frac{n-k+1}{n}\sigma_{1,1}\sigma_{k-1,k-1}\right) .
 \end{equation}
 Moreover, if $k\geq3$ and equality holds, then $\tf P\rv_{TM}=0$.

 From the assumption~\eqref{eqn:obata_assumption}, we may fix a positive constant $y\in\bR$ such that
 \[ H \leq y \leq (k+1)H . \]
 By Lemma~\ref{lem:cKilling}, it holds that
 \[ \oint_{S^n} Hu\sigma_{k,k} = \oint_{S^n} (H-y)u\sigma_{k,k} . \]
 Since $H\leq y$ and $\sigma_{1,1},\sigma_{k-1,k-1}\geq0$ (see the proof of Lemma~\ref{lem:signs}) we obtain
 \begin{equation}
  \label{eqn:cKilling_application}
  \oint_{S^n} Hu\sigma_{k,k} \leq \oint_{S^n} (H-y)u\left(\sigma_{k,k}-\frac{n-k+1}{nk}\sigma_{1,1}\sigma_{k-1,k-1}\right)
 \end{equation}
 with equality if and only if $(H-y)\sigma_{k-1,k-1}=0$.

 Combining~\eqref{eqn:ibp2} and~\eqref{eqn:cKilling_application} yields
 \begin{equation}
  \label{eqn:conclusion}
  0 \leq \oint_{S^n} \left((k+1)H-y\right)\left(\sigma_{k,k}-\frac{n-k+1}{nk}\sigma_{1,1}\sigma_{k-1,k-1}\right) u .
 \end{equation}
 Since $y\leq(k+1)H$, we conclude from Lemma~\ref{lem:signs} that equality holds in~\eqref{eqn:conclusion}.  If $k\geq3$, the equality case of~\eqref{eqn:ibp2} implies that $\tf P\rv_{TM}=0$.  If $k=2$, we conclude from the equality case of~\eqref{eqn:cKilling_application} that $(H-y)\sigma_{1,1}=0$ almost everywhere, while it is clear from Lemma~\ref{lem:signs} and~\eqref{eqn:conclusion} that $(3H-y)(\sigma_{2,2}-\frac{n-1}{2n}\sigma_{1,1}^2)=0$.  Let
 \[ \mathcal{E} = \left\{ p\in M \suchthat \sigma_{2,2}(p) = \frac{n-1}{2n}\sigma_{1,1}^2(p) \right\} . \]
 Suppose that $p\in S^n\setminus\mathcal{E}$.  Then $3H(p)=y$, and hence $\sigma_{1,1}(p)=0$.  Since
 \[ 3n^2H_2 = 3nH\sigma_{1,1} + H^3 \geq H^3 , \]
 we see that $H$ achieves its maximum at $p$, contradicting $3H(p)=y\geq\sup H$.  Thus $S^n=\mathcal{E}$, and hence, by~\eqref{eqn:cKilling}, we again conclude that $\tf P\rv_{TM}=0$.  Since $S^n\subset S_+^{n+1}$ is umbilic and $\tf P\rv_{TM}=0$, the Gauss--Codazzi equations imply that $g\rv_{S^n}$ is Einstein.  As $g\rv_{S^n}$ is conformal to a round metric, it is itself a round metric.  Since the round metric on $S^n$ is the boundary metric of the flat metric in the $(n+1)$-ball, we conclude from the uniqueness assertion of Theorem~\ref{thm:inequality} that $g$ is flat.
\end{proof}

\bibliographystyle{abbrv}
\bibliography{../bib}
\end{document}